\documentclass[a4paper,11pt,reqno]{article}
\usepackage[english]{babel}
\usepackage{cite}
\usepackage{amsfonts}
\usepackage{latexsym}
\usepackage{amsthm}
\usepackage{amsopn}
\usepackage{amsmath}
\usepackage[all]{xy}
\usepackage{verbatim}
\usepackage{amssymb}
\usepackage{mathrsfs}
\usepackage{float}
\usepackage[mathscr]{eucal}
\usepackage{marginnote}
\usepackage[active]{srcltx}
\usepackage{fullpage}
\usepackage{hyperref}
\usepackage{manyfoot}
\usepackage{amscd}
\usepackage[dvips]{graphics}
\usepackage[dvips]{graphicx}
\usepackage{amscd}
\usepackage{color}
\usepackage{cite}
\usepackage{tikz}
\usepackage{pgfplots}
\pgfplotsset{compat=1.15}
\usepackage{mathrsfs}
\usetikzlibrary{arrows}
\usepackage{textcomp}

\def\new{\textnormal{new}}
\def\tra{\textnormal{tra}}
\def\alg{\textnormal{alg}}

\usepackage[inline]{enumitem}
\setlist{noitemsep,nosep,listparindent=\parindent}
\setlist[itemize]{label=\guillemotright}
\setlist[enumerate,1]{ref=\thesubsection.\arabic*}
\setlist[enumerate,2]{label=\alph*.,ref=\theenumi.\alph*}

\title{Automorphisms of a family of surfaces with $p_g=q=2$ and $K^2=7$}
\author{Matteo Penegini and Roberto Pignatelli}
\date{}

\def\QQl{\QQ_\ell}

\newcommand\Aut{\text{\rm Aut}}
\newcommand\Pic{\text{\rm Pic}}

\newcommand\OO{{\mathcal{O}}}

\newcommand{\CC}{{\mathbb C}}

\newcommand{\HH}{{\mathbb H}}

\newcommand{\PP}{{\mathbb P}}
\newcommand{\QQ}{{\mathbb Q}}

\newcommand{\ZZ}{{\mathbb Z}}

\newtheorem{thm}{Theorem}[section]
\newtheorem{lem}[thm]{Lemma}
\newtheorem{cor}[thm]{Corollary}
\newtheorem{prop}[thm]{Proposition}

\theoremstyle{definition}
\newtheorem{defin}[thm]{Definition}

\newtheorem{rem}[thm]{Remark}
\theoremstyle{remark}

\newcommand{\oo}{\mathcal{O}}

\def\Spec{\textnormal{Spec}}
\def\Mot{\textnormal{Mot}}
\def\AbMot{\textnormal{AbMot}}
\def\HH{\mathrm{H}_{\textnormal{mot}}}
\def\ChMot{\mathcal{M}_{\textnormal{rat}}}
\def\SmPr{\textnormal{SmPr}}
\def\opp{\textnormal{op}}

\def\AJ{\textnormal{AJ}}
\def\Spec{\textnormal{Spec}}
\def\Mot{\textnormal{Mot}}
\def\AbMot{\textnormal{AbMot}}
\def\HH{\mathrm{H}_{\textnormal{mot}}}
\def\GB{G_{\textnormal{MT}}}
\def\Gl{G_\ell}
\def\Glc{\Gl^\circ}

\def\HB{H}
\def\Hl{H_\ell}

\newcommand{\ra}{\rightarrow}

\def\new{\textnormal{new}}
\def\tra{\textnormal{tra}}
\def\alg{\textnormal{alg}}

\begin{document}
\maketitle


\begin{center}
{\it In memoria di Gianfranco, un amico e un maestro}
\end{center}
\begin{abstract} 
We compute the automorphism group of all the elements of a family of surfaces of general type with  $p_g=q=2$ and $K^2=7$, originally constructed by C. Rito in \cite{rito}. We discuss the consequences of our results towards the Mumford-Tate conjecture.
\end{abstract}

\Footnotetext{{}}{\textit{2020 Mathematics Subject Classification}:
14J29, 14J10, 14B12}

\Footnotetext{{}} {\textit{Keywords}: Surface of general type,
Albanese map, Automorphisms of varieties}

\Footnotetext{{}} {\textit{Version:} \today}
\section{Introduction}
 
The classification of general type surfaces with low numerical invariants is unanimously considered a very difficult problem to tackle. It is already difficult to construct new surfaces with low numerical invariants. Therefore, as soon as new examples are found, it is natural to test the famous conjectures on them. This short note stems from the question whether it is possible to verify the Mumford-Tate conjecture for surfaces $S$ of general type with $p_g=q=2$ and $K^2=7$ first constructed by Rito in \cite{rito} and later studied by the authors in \cite{MatteoRoberto}. 

These surfaces $S$ are obtained as a generically finite double covering of an abelian surface $A$, which turns out to be the Albanese variety of $S$, branched along a curve with a singular point of type $(3,3)$ and no other singularities. These surfaces give rise to three disjoint open subsets in the  Gieseker moduli space $\mathcal{M}^{\mathrm{can}}_{2, \, 2, \,7}$  which are all irreducible, generically smooth of dimension 2, that we shall denote by $\mathcal{M}_1$,  $\mathcal{M}_2$ and  $\mathcal{M}_4$. 

The Mumford-Tate conjecture for surfaces is still an open problem and only in very few examples it has been verified as true, see for example \cite{Moonenh1}. A strategy to prove the conjecture for surfaces with $p_g=q=2$ and of maximal Albanese dimension is outlined in the article \cite{CoPe20}. The strategy reduces to finding  geometric quotients $X$ of $S$ that are K3 surfaces whose weight 2 Hodge structure is a sub-Hodge structure of the weight 2 Hodge structure of $S$ \emph{orthogonal} to the the sub-Hodge structure coming from the Albanese surface of $S$. This strategy proved to be very successful in many cases, see  \cite{CoPe20}. The first question toward exploiting the strategy is to calculate the automorphism of the surfaces $S$ and then classify all possible quotients. This is the content of the main theorem of this note.

\begin{thm}\label{theo_main}  The automorphism group of the surface $S$ of general type with $p_g=q=2$ and $K^2=7$  constructed by Rito in \cite{rito} is a product of cyclic groups
\[Aut(S) \cong G \times \ZZ/2\ZZ = G \times \langle \sigma \rangle,
\]
where $S/ \langle \sigma \rangle \cong A=Alb(S)$, while 
\begin{enumerate}
\item $G$ is trivial if $[S]  \in \mathcal{M}_4$;
\item $G \cong \ZZ/4$ if $[S]  \in \mathcal{M}_1$ and  satisfies condition $1$ of Proposition \ref{prop:AutA}
\item $G \cong \ZZ/2$  in all the other cases.
\end{enumerate}
\end{thm}

The second step of the strategy is to identify the quotients. We have at once 

\begin{cor}\label{cor_main} For all $H \leq \textrm{Aut}(S)$,  the quotients $X = S/H$ are irregular surfaces, i.e., $q(X) \geq 1$. 
\end{cor}

This corollary tells us that in order to prove the Mumford-Tate conjecture for these surfaces a new strategy is needed. 

Now, let us explain  the way in which this paper is organized.

In the second section we recall the construction of the surfaces $S$ with the calculation of the invariants.  The third section is devoted to the proof of the Main Theorem. The section four contains the  calculation of the invariants of the quotient surfaces. Finally we include a section five where it is explained the strategy  to prove the conjecture for surfaces with $p_g=q=2$ and of maximal Albanese dimension.

\textbf{Acknowledgments.} 
The first author was partially supported by GNSAGA-INdAM, by PRIN 2020KKWT53 003 - Progetto: \emph{Curves, Ricci flat Varieties and their Interactions} and by the DIMA - Dipartimento di Eccellenza 2023-2027. 
The second author was partially supported by supported by the "National Group for Algebraic and Geometric Structures, and their Applications" (GNSAGA - INdAM) and by the European Union under NextGenerationEU, PRIN 2022 Prot. n. 20223B5S8L.

\medskip

\textbf{Notation and conventions.} We work over the field
$\mathbb{C}$ of complex numbers. 
By \emph{surface} we mean a projective, non-singular surface $S$,
and for such a surface $K_S$ denotes the canonical
class, $p_g(S)=h^0(S, \, K_S)$ is the \emph{geometric genus},
$q(S)=h^1(S, \, K_S)$ is the \emph{irregularity} and
$\chi(\mathcal{O}_S)=1-q(S)+p_g(S)$ is the \emph{Euler-Poincar\'e
characteristic}.
\section{The surfaces}\label{Sec_theConstruction}

In this section we report, for the convenience of the reader, the construction of the families of 
surfaces under consideration. We use the notation of \cite{MatteoRoberto}, the main result in this direction is the following one.

\begin{prop}{\cite{PP}}\label{prop_invariantiRito}
Let $A$ be an Abelian surface.
Assume that $A$ contains a reduced curve whose class is 2-divisible in $\Pic(A)$, whose self intersection is $16$, with a unique singular point of type $(3,3)$ and no other singularity. Then there exists a generically finite double cover $S \rightarrow A$ branched along this curve. Moreover,  the numerical invariants of $S$ are  $p_g(S)=q(S)=2$ and $K_S^2=7$.
\end{prop}

In \cite{rito} and later in \cite{MatteoRoberto} the existence of the abelian surface $A$, that has the properties as in Proposition \ref{prop_invariantiRito}, is proved. In particular, it is also shown that the double cover coincides with the Albanese map, hence $A$ is the Albanese variety associated to $S$, we denote it by
\[
\alpha \colon S \rightarrow \textrm{Alb}(S)=A.
\]
We can be more precise,  $A$  is isogenous to a product of two elliptic curves $T_1$ and $T_2$. We denote by 
\[
\iota \colon A \rightarrow T_1 \times T_2
\]
the isogeny, which is of degree 2. Clearly,  $A$ carries a (1,2)-polarization $L$ which  is a pull-back of a (product) principal polarization via the isogeny $\iota$. In addition, on $A$ we have  two elliptic fibrations $f_j \colon A \rightarrow T_j$ with fibres  $\Lambda_i$ with $\Lambda_i$ isogenous to $T_i$ by a degree two isogeny for $i,j \in \{1,2\}$. Notice that the isogeny is given by the restriction of $\iota$ to the fibres. 

The branching locus of $\alpha$ is an effective divisor with two irreducible components
\begin{equation}\label{eq_CtL}
C_1+t \in |2L|,
\end{equation}
where $C_1=f_2^{-1}(b_1)$ is an element of $|\Lambda_2|$, with $b_1 \in T_2$. While, $t$ is a curve of geometric genus $3$ with a tacnode tangent  to $C_1$ at a point  $p$. The situation is exemplified in the following  Figure \ref{Fig1}.

\begin{figure}[h]
\centering
\includegraphics[width=0.55\linewidth]{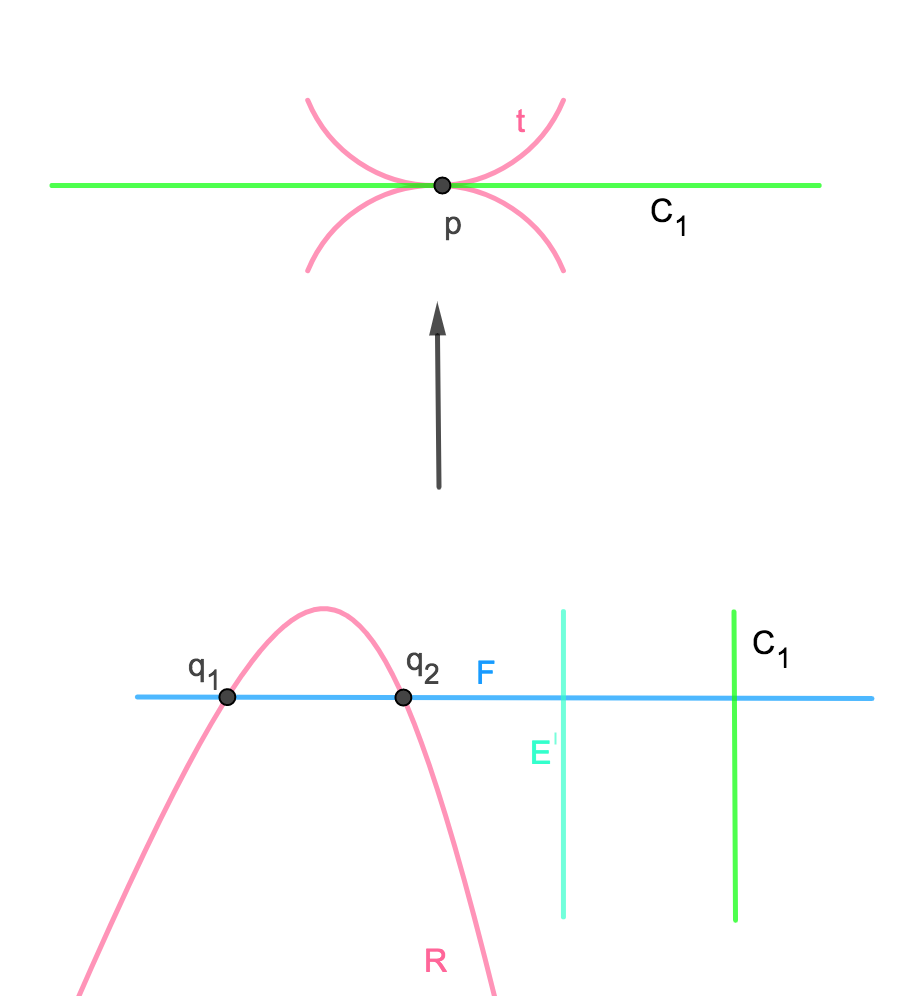}
\caption{The Branch Divisor of $\alpha$ and $\beta$}\label{Fig1}
\end{figure}

We deduce that
\begin{align}\label{eq_intersecions}
C_1^2=&0,&
t^2=&8,&
C_1t=&4,
\end{align}
notice that $(C_1+t)^2=16$. 

Notice that the branch locus is singular in $p$. Therefore, to get a smooth surface $S$ as a generically finite double cover of $A$ branched along $C_1+t$  we have to blow up the point $p$ (see Figure \ref{Fig1}) first. 

\begin{enumerate}
\item First, we resolve the singularity in $p$. To do that, we need to blow up $A$ twice, first  in $p$ and then in a point infinitely near to $p$. Let us denote these two blow ups by 
\[ B' \stackrel{\sigma_4}{\longrightarrow} B \stackrel{\sigma_3}{\longrightarrow} A.
\]
On $B'$, let us denote by $F$ the exceptional divisor relative to $\sigma_4$, by $E'$ the strict transform of the exceptional divisor $E$ relative to $\sigma_3$, by $C_1$ the strict transform of $C_1$ and, finally, by $R$ the strict transform of $t$ (see \verb|Figure| \ref{Fig1}). 

In addition, one gathers  the following information:
$E'\cong \PP^1$ and $(E')^2=-2$, $F\cong \PP^1$ and $F^2=-1$,  $g(C_1)=1$  and $C_1^2=-2$. 

\item Second, we consider a double cover of $\beta\colon S'\longrightarrow B'$ ramified over $R+C_1+E'$ (this is even since $t+C_1$ is even on $A$). The surface $S'$ is a surface of general type, not minimal. Indeed, it contains a $-1$-curve, which is $\hat{E}=\beta^{-1}(E')$. The ramification divisor is denoted $\hat{R}+\hat{C_1}+\hat{E}$. Notice that $\hat{C_1}$ has genus $1$ and $\hat{C_1}^2=-1$.

\item Finally, to get $S$ we contract the $-1$-curve $\hat{E}$.
\end{enumerate}

We can summarize the construction of $S$ with the following diagram.
  \[
\begin{xy}
\xymatrix{
S'   \ar@{->}[d]    \ar@{->}[r]^{\beta}   &  B' \ar@{->}[d]^{\sigma_4} \\
S   \ar@{->}[dr]_{\alpha}  &  B \ar@{->}[d]^{\sigma_3} \\ 
& A 
 }
\end{xy}
\]

Moreover, the point $p$ is a [3,3] point, which is not a negligible singularities. A [3,3] point is a pair  $(x_1, \, x_2)$ such that $x_1$ belongs to the first infinitesimal neighborhood of $x_2$ and both are triple points for the curve. 
Thus, we may calculate the invariants of $S$  by using the formulae in \cite[p. 237]{BHPV03}. In those formulas $x_2$ counts as a triple point (so $m_2=1$) and $x_1$ as a quadruple point (so $m_1=2$). Then 
\begin{equation} \label{eq.res.can}
2=2 \chi(\oo_{S'})=L^2- \sum_{i=1}^2 m_i(m_i-1), \quad
6=K_{S'}^2=2L^2-2 \sum_{i=1}^2  (m_i-1)^2.
\end{equation}

Finally, once we contract the $-1$-curve on $S'$, we obtain
hence $K_S^2=7$ and $\chi(S)=1$.


\medskip

Considering the Abelian varieties $A, T_1, T_2, T_1 \times T_2$ we choose the following points as neutral elements:
\begin{align*}
p&\in A,&
a_3:=&f_1(p)\in T_1,&
b_1:=&f_2(p)\in T_2,&
(a_3,b_1) &\in T_1 \times T_2.
\end{align*}
With this particular choice $\iota, f_1,f_2$ are  homomorphism of groups too.

The remaining  $2$-torsion points on each  elliptic curve $T_j$ will be denoted by
\begin{align*}
a_1,a_2,a_4 &\in T_1[2],&
b_2,b_3,b_4 &\in T_2[2].
\end{align*}
This yields  $\iota_* \OO_A^- \cong \OO_{T_1}(a_4-a_3) \boxtimes \OO_{T_2}(b_2-b_1) $ , where $\iota_* \OO_A^- $ is the anti-invariant part of $\iota_* \OO_A$,  see \cite[Lemma 3.4]{MatteoRoberto} for a detailed proof.

\begin{rem}\label{rem_the3cases}
Furthermore in \cite{MatteoRoberto} it is proved that
\[ f_1^*(a_4+a_3) + f_2^*(b_3+b_1) \in |2L|,
\]
whence
\[L \cong f_1^*(\bar{a})  \otimes f_2^*(\bar{b})
\]
where $\bar{b}$ is a 4-torsion point such that $\bar{b} \oplus \bar{b} \neq b_2$.  While for $\bar{a}$ we have three possible choices by \cite[Proposition 3.6]{MatteoRoberto} 

\begin{enumerate}
\item $\bar{a}=a_3$ (in this case $\bar{b} \oplus \bar{b}=b_4$);
\item $\bar{a}$  is a 2-torsion point such that $\bar{a} \neq a_4$ (in this case $\bar{b} \oplus \bar{b}=b_4$);
\item $\bar{a}$ is a 4-torsion point such that $\bar{a} \oplus \bar{a} = a_4$ (in this case $\bar{b} \oplus \bar{b}=b_3$).
\end{enumerate}

As just remarked all the choices are possible and to each choice corresponds a different irreducible component of the Gieseker moduli space $\mathcal{M}^{\textrm{can}}_{2,2,7}$ of the canonical models of the surfaces of  general type with $p_g = q = 2$ and $K^2 = 7$. We shall denote these components, following \cite[Definition 3.7]{MatteoRoberto}, by $\mathcal{M}_i \subset \mathcal{M}^{\textrm{can}}_{2,2,7}$ with $i \in \{1,2,4\}$ respectively. Note that the index $i$ equals the order of $\bar{a}$ as torsion point.
\end{rem}

\section{The automorphisms of the Rito's surfaces}\label{sec_klein} 

Consider the abelian variety $A$. We know that there is an isogeny of degree $2$ onto a product of elliptic curves $T_1 \times T_2$.

By taking the universal covers, we can write $T_j:=\CC/\lambda_j$ where the $\lambda_j \cong \ZZ^2$ are lattices so that the origin maps to $a_3$  respectively $b_1$. 
We choose generators $\bar{e}_1,\bar{e}_2$ of $\lambda_1$, $\bar{e}_3,\bar{e}_4$ of $\lambda_2$ so that $\frac{\bar{e}_1}2$ maps to $a_1$,  $\frac{\bar{e}_2}2$ maps to $a_2$,  $\frac{\bar{e}_3}2$ maps to $b_3$,  $\frac{\bar{e}_4}2$ maps to $b_4$.

So, in $\CC^2$ -coordinates we have
\begin{align*}
e_1=&(\bar{e}_1,0),&
e_2=&(\bar{e}_2,0),&
e_3=&(0,\bar{e}_3),&
e_4=&(0,\bar{e}_4).&
\end{align*}
Since the universal cover of $T_1 \times T_2$ factors through the isogeny $\iota$,
we obtain $A=\CC^2/\lambda$ where $\lambda$ is a sublattice  of index $2$ of the lattice
\[
\lambda_1 \oplus \lambda_2 =\left\{ \sum t_i e_i | t_i \in \ZZ \right\}
\] 

\begin{lem}\label{lem:sublattice}
The lattice $\lambda$ is the sublattice of $\lambda_1 \oplus \lambda_2$ of the elements whose sum $\sum t_i$ is even.
\end{lem}
\begin{proof}
We set $\tilde{\lambda}:= \left\{ \sum m_j e_j | \sum m_j \text { is even}\right\}$, $\tilde{A}:= \CC^2/\tilde{\lambda}$. The inclusion $\tilde{\lambda} \subset \lambda_1 \oplus \lambda_2$ induces an isogeny $\tilde{\iota} \colon \tilde{A} \rightarrow T_1 \times T_2$ of degree $2$.

An isogeny of degree $2$ is determined by the anti-invariant part of the direct image of the trivial bundle, that is a generator of the kernel of the pull-back map among the Picard groups. So we only need to prove that $\tilde{\iota}^* \left( \OO_{T_1}(a_4-a_3) \boxtimes \OO_{T_2}(b_2-b_1) \right)$ is the trivial sheaf of $\tilde{A}$.

I

Let $\lambda'_1 \subset \lambda_1$ be the index $2$ sublattice of the elements of the form $m_1\bar{e}_1+ m_2 \bar{e}_2$ with $m_1+m_2$ even. In the same way, let $\lambda'_2 \subset \lambda_2$ be the index $2$ sublattice of the elements of the form $m_3\bar{e}_3+ m_4 \bar{e}_4$ with $m_3+m_4$ even.  These define an isogeny of degree $2$, $\tilde{\iota}_i \colon T'_i= \CC/\lambda'_i \rightarrow T_i$, for $i=1,2$. We can derive the following  commutative diagram  
\begin{equation}\label{diag:isogenies}
\xymatrix{ & T'_1 \times T'_2 \ar[dr] \ar[dl] \ar[d] \\
T_1' \ar[dd]_{\tilde{\iota}_1 }& \tilde{A}  \ar[d]^{\tilde{\iota}} &  T'_2 \ar[dd]^{\tilde{\iota}_2} \\ 
& T_1 \times T_2  \ar[dr] \ar[dl]& \\ 
T_1 && T_2
}
\end{equation}


Notice that  $\tilde{\iota}_2^* \OO_{T_2}(b_2-b_1)$ is trivial  on $ \CC/\lambda'_2$. This is standard; it can be show for example as follows.

The point $b_1$ pulls back to the sum of two points, the classes modulo $\lambda'_2$ of $0$ and $\bar{e}_3$. The point $b_2$ pulls back to the sum of the classes of $\frac12 (\bar{e}_3+\bar{e}_4)$ and  $\bar{e}_3+\frac12 (\bar{e}_3+\bar{e}_4)$. Since 
\[
\left(\frac12 (\bar{e}_3+\bar{e}_4) + \bar{e}_3+ \frac12 (\bar{e}_3+\bar{e}_4) \right) - (0+\bar{e}_3) =\bar{e}_3+\bar{e}_4\in \lambda'_2
\]
the two divisors of degree $2$ we have obtained are linearly equivalent.

In the same way we can prove that   $\tilde{\iota}_1^* \OO_{T_1}(a_4-a_3)$ is trivial on $ \CC/\lambda'_1$.

The 2-torsion line bundles on  $T_1 \times T_2$ that pull back to trivial bundles on $T'_1 \times T'_2$ are the  line bundles: $\OO_{T_1}(a_4-a_3) \boxtimes \OO_{T_2}$, $\OO_{T_1} \boxtimes  \OO_{T_2}(b_2-b_1)$  and   $\OO_{T_1}(a_4-a_3) \boxtimes \OO_{T_2}(b_2-b_1)$. Exactly one of them pulls back to the trivial line bundle to $\tilde{A}$.

We can conclude the proof observing that if $\tilde{\iota}^* \left( \OO_{T_1}(a_4-a_3) \boxtimes \OO_{T_2} \right)$ were trivial on $\tilde{A}$ than this would imply that there were a fibration  form $\tilde{A}$  onto $T'_1$  and  this is absurd.  In the same way we exclude the case $\OO_{T_1} \boxtimes  \OO_{T_2}(b_2-b_1)$.
\end{proof}

We will need the following general result for an abelian surface with a $(1,2)$-polarization, the proof of which can be found in \cite[Section 1.2]{barth}.
\begin{rem}\label{rem:dueriducibili}
The linear system $|L|$ contains exactly two reducible divisors union of elements respectively of $\Lambda_1$ and $\Lambda_2$, the curves $f_1^*\bar{a}+f_2^* \bar{b}$ and $f_1^*(\bar{a} \oplus a_4)+f_2^* (\bar{b} \oplus b_2)$. 
\end{rem}

Since the Albanese morphism $\alpha\colon S \longrightarrow A$ has degree $2$, it determines an involution $\sigma \colon S \rightarrow S$ that is central in $\Aut\ S$ and an exact sequence
\[
0
\rightarrow
\ZZ/2\ZZ=\langle \sigma \rangle
\rightarrow
\Aut \ S
\rightarrow
G
\rightarrow
0
\]  
where $G$ is the group of the self-biholomorphisms $\varphi \colon A \rightarrow A$ such that
\begin{enumerate}
\item $\varphi^*L=L$;
\item $\varphi (C_1+t)=C_1+t$, equivalently $\varphi (C_1)=C_1$, $\varphi (t)=t$, $\varphi(p)=p$.
\end{enumerate}

We have written an isomorphism $A\cong \CC^2/\lambda$ where the point $p$ is the image of the origin of $\CC^2$.
From $\varphi(p)=p$ it follows that the elements of $G$ are automorphisms of $A$ as a group with the group structure induced by $\CC^2$.
Now, we can see that one of the above assumption is not necessary, indeed we have the following lemma.

\begin{lem}
The group $G$ is the group of  the automorphisms of the Abelian variety $A$ preserving the group structure induced by the identification $\CC^2/\lambda =A$ such that
 $\varphi^*L=L$ and  $\varphi (C_1)=C_1$.
\end{lem}

\begin{proof}

The only nontrivial thing to prove is that, if  $\varphi \colon A \rightarrow A$  is a  group automorphism such that
 $\varphi^*L=L$ and  $\varphi (C_1)=C_1$, then  $\varphi(t)=t$.  Notice that by hypothesis and equation \eqref{eq_CtL} we have that $t$ is linearly equivalent to $\varphi(t)$.

By \eqref{eq_intersecions}, the intersection number between $t$ and $C_1$ is $4$. More precisely, $t$ cuts on $C_1$ the divisor $4p$. Since $\varphi(p)=p$ and $\varphi(C_1)=C_1$, $\varphi(t)$ cuts $4p$ on $C_1$ as well. 

Assume by contradiction $\varphi(t)\neq t$, then the functions defining them span a subspace $V \subset H^0(A,t)$ of dimension $2$. By what we just said $t$ and  $\varphi(t)$ cut on $C_1$ the same divisor $4p$, therefore the restriction map $\rho\colon V \longrightarrow H^0(C_1, t|_{C_1})$ has rank $1$. By dimension count the kernel of $\rho$ is a one dimensional subspace generated by  say  $s$. Then $C_1 \subset \{s=0\}$, let us call $D_1$ the residue curve.  By definition $D_1$ is an effective divisor  in $|t-C_1|=|2L-2C_1|$.

Notice that $L$ is numerically equivalent to $\Lambda_1+\Lambda_2$ while  $C_1$ is numerically equivalent to $\Lambda_2$. Since $D_1$ is linearly equivalent to  $2L-2C_1$ if follows that it is numerically equivalent to $2\Lambda_1$, hence the intersection product $D_1 \cdot \Lambda_1 =0$. Since any element in $|\Lambda_1|$ is irreducible we have that $D_1$ is union of two of such elements. Let us denote these two elements them by $A$ and $B$ and we get numerically $D_1=A+B$. 

Let $D$ be a further element of $|\Lambda_1|$. Then $D$ is isomorphic to $\CC/\lambda_2'$. More precisely,  we have the following diagram

\begin{equation}\label{diag:isogenies}
\xymatrix{
 \tilde{A} \ar[r]^{\tilde{\iota}}  & T_1 \times T_2 \ar[d] \\
 \CC/\lambda'_2 \ar[u]^{\xi}  \ar[r]^{\tilde{\iota}_2} & T_2
}
\end{equation}
where $\xi$ maps isomorphically $ \CC/\lambda'_2$ onto $D$. Notice that $\xi$, being an isomorphism,  allows us to see $D$ as a degree $2$ \'etale cover of $T_2$ via the composition with the isogeny $\tilde{\iota}_2$.  Since $A$ or $B$ restricted to $D$ are trivial so is the restriction of $D_1$. A fortiori  the restrictions to $D$ of $2L$ and $2C_1$ are linearly equivalent, which means that the restriction to $D$ of $L$ and $C_1$ differ by $2$-torsion. 

The restriction to $D$ of $C_1$ is $\tilde{\iota}_2^* b_1$ since $b_1$ is $f_2(p)$. Moreover the restriction to $D$ of $L$   is $\tilde{\iota}_2^*  \bar{b}$ by  Remark \ref{rem_the3cases}. Hence the $4$-torsion point $ b_1 - \bar{b}$ in $T_2$ lifts to a $2$-torsion point  in $D$. Thus  the line bundle $\OO_{T_2}(2(b_1 - \bar{b}))$ is the only $2$-torsion on $T_2$ that lifts to the trivial bundle on $D$. This implies (see also proof of Lemma \ref{lem:sublattice}) that $\bar{b} \oplus  \bar{b}=b_2$ but this is absurd because by  Remark \ref{rem_the3cases} we have either $\bar{b} \oplus  \bar{b}=b_3$ or $\bar{b} \oplus  \bar{b}=b_4$.
\end{proof}

The action of $G$ on $A$ may be uniquely lifted to an action on $\CC^2$ fixing the origin, so representing $G$ as a finite subgroup of the linear group $GL_2(\CC)$ of the matrices preserving the lattice $\lambda$. We will then write $\varphi$ as a matrix
\[
\varphi=
\begin{pmatrix}
\varphi_{11}&
\varphi_{12}\\
\varphi_{21}&
\varphi_{22}
\end{pmatrix}
\]

\begin{lem}\label{lem:diagonal}
$G$ is a group of diagonal matrices. In other words $\varphi_{12}=\varphi_{21}=0$.
\end{lem}
\begin{proof}
 We consider  $\CC^2$ with the natural coordinates $(x,y)$ given by the construction, so that the line $x=0$ is the connected component through the origin of the preimage of $C_1$ in $\CC^2$. Since $\varphi(C_1)=C_1$ the matrix $\phi$ preserves $x=0$, so $\varphi_{12}=0$. 

Now we use that $\varphi^*L=L$. Since by Remark \ref{rem:dueriducibili}, $|L|$ contains exactly two reducible divisors, $\varphi$ either preserves or exchange them.  The preimage of both divisors on $\CC^2$ is a union of countably many "horizontal" lines $y=c$ and countably many vertical lines $x=c$. However, if $\varphi_{21}\neq 0$ any horizontal line  is mapped to a line that is neither horizontal nor vertical, a contradiction. 
\end{proof}

In particular $\varphi$ is given by two roots of the unity $\varphi_{jj}\in \CC$ giving automorphisms of the two elliptic curves: each $\varphi_{jj}$ gives an automorphism of $T_j$. 

We will now need the following well known facts on automorphisms of elliptic curves see for instance \cite[Section III.10]{SilAEC}.
\begin{lem}\label{lem:actions}
Let $\omega$ be a nontrivial automorphism of an elliptic curve of order $n$. Then $n=2,3,4$ or $6$. Moreover
\begin{enumerate}
\item for every $4-$torsion point $p\in T$, $\omega(p)\neq p$;
\item if $n=3,6$, then for every $2-$torsion point $p\in T$, $\omega(p)\neq p$;
\item if $n=4$ then there is exactly one $2-$torsion point $p\in T$ such that $\omega(p)=p$;
\item if $n=2$ then $\omega(p)=p$ for all $2-$torsion points $p\in T$.
\end{enumerate}
\end{lem}

With these facts in mind we can prove

\begin{prop}\label{prop:AutA} We have the following possibilities for the group $G$ according to the cases in Remark \ref{rem_the3cases}.
\begin{itemize}

\item In case $1$: if $T_1$ has an automorphism of order $4$, $G$ is cyclic of order $4$ generated by the automorphism given by $\varphi_{11}=i$, $\varphi_{22}=1$, that is
\[
(x,y)\mapsto (i x,y).
\]  

\item In case $3$:  $G$ is the trivial group of order $1$.

\item In the remaining cases, $G$ is a cyclic group of order $2$ generated by the automorphism given by $\varphi_{11}=-1$, $\varphi_{22}=1$, the involution 
\[
(x,y)\mapsto (-x,y)
\]  
\end{itemize}
\end{prop}
\begin{proof}
By Lemma \ref{lem:diagonal} $\varphi$ acts on the fibrations $f_j \colon A \rightarrow T_j$ acting on the codomain by $\varphi_{jj}$.

Since by Remark \ref{rem:dueriducibili}, $|L|$ contains exactly two reducible divisors, $f_1^*\bar{a}+f_2^* \bar{b}$ and $f_1^*(\bar{a} \oplus a_4)+f_2^* (\bar{b} \oplus b_2)$, $\varphi$ either preserves or exchange them. So, on $T_1 \times T_2$, the matrix $\varphi$ maps $(\bar{a},\bar{b})$   either to $(\bar{a},\bar{b})$ or to $(\bar{a} \oplus a_4, \bar{b} \oplus b_2)$.

We now show $\varphi_{22}=1$. In fact in both cases $\varphi_{22}^2(\bar{b})=\bar{b}$. Since 
$\bar{b}$ is a 4-torsion point, by Lemma \ref{lem:actions}, part 1, $\varphi_{22}^2=1$. Moreover, if $\varphi_{22} \neq 1$  (so  $\varphi_{22} = -1$) $\varphi_{22}(\bar{b})=\bar{b} \oplus b_2$ that implies $\bar{b} \oplus \bar{b} = b_2$, a contradiction. So $\varphi_{22}=1$.

As a first consequence, $\varphi_{11}(\bar{a})=\bar{a}$.

Now recall that the matrix $\varphi$ preserves the lattice $\lambda$. Since $\varphi_{22}=1$, then  $\varphi_{11}$ preserves $\lambda_1$ and the index $2$ sublattice $\lambda'_1 \subset \lambda_1$ of the elements of the form $m_1\bar{e}_1+ m_2 \bar{e}_2$ with $m_1+m_2$ even. This implies $\varphi_{11} (\bar{e}_1+  \bar{e}_2)-(\bar{e}_1+  \bar{e}_2) \in 2\lambda_1$. Dividing by $2$ we obtain $\varphi_{11}(a_4)=a_4$.

Now we distinguish the three cases according to Remark \ref{rem_the3cases}.

In case $3$, $\bar{a}$ is a $4-$torsion point. Then by  $\varphi_{11}(\bar{a})=\bar{a}$ and Lemma \ref{lem:actions}, part 1, $\varphi_{11}=1$.

In case $2$, $\bar{a}$ and $a_4$ are distinct $2-$torsion points fixed by $\varphi_{11}$.  Then by Lemma \ref{lem:actions}, part 2 and 3, $\varphi_{11}$ has order $1$ or $2$, so $\varphi_{11}=\pm 1$. On the other hand the map $(x,y)\mapsto(-x,y)$ preserves $\lambda$ so it defines an automorphism of $A$ that defines an element of $G$.

Finally, case $1$. In this case $\varphi_{11} (\bar{a})=\bar{a}$ holds indipendently by the choice of the complex number of $\varphi_{11}$. Still we have the condition $\varphi_{11}(a_4)=a_4$ that by Lemma \ref{lem:actions}, part 2, forces $\varphi_{11}^4=1$.
If $T_1$ has no automorphisms of order $4$, then we obtain $\varphi_{11}=\pm 1$ and we conclude as in case 2. Else analogous argument shows that $(x,y)\mapsto (i x,y)$ generates $G$.
\end{proof}

Then we can compute $\Aut(S)$.

\begin{prop}\label{prop_autS}
\[Aut(S) \cong G \times \ZZ/2\ZZ
\]
\end{prop}
\begin{proof}
Choose any element $g\in G$. Then by \ref{prop:AutA} $g$ acts as  $(x,y) \mapsto (k x,y)$ in the coordinates considered there, where $k$ is a complex number with $k^4=1$. 
In those coordinates $C_1$ is defined by $y$. 

Let $Z$ be the finite double cover of $A$ birational to $S$, and let $q\in Z$ be the unique point over $p=(0,0)$.
In a neighbourhood of $q$, $Z$ has equation $z^2=f(x,y)$ where $f$ is an equation of the branch locus, that is geometrically $g$-invariant.
Then there is a constant $c$ such that  $g^*f=cf$.
From $f(x,y)=y^3+O(4)$ we deduce $c=1$ and $g^*f=f$.

The involution on $Z$ induced by the Albanese morphism of $S$ acts near $q$ as $(x,y,z) \rightarrow (x,y,-z)$;
$g$ acts $(x,y) \rightarrow (kx,y)$, so the liftings of $g$ act as $(x,y,z) \mapsto (kx,y, \pm z)$. So the liftings acting locally trivially on the variable $z$ form a 
splitting map $G \rightarrow \Aut(S)$ mapping to a subgroup that commutes with the Albanese involution.
\end{proof}

\section{The quotients of $S$}\label{sec_quot}

In this brief section we shall prove Corollary \ref{cor_main}. By Proposition \ref{prop_autS}  $\Aut(S) =\langle \sigma \rangle \times G$ let us consider $H \leq \Aut(S)$, then we have the following diagram
  \[
\begin{xy}
\xymatrix{
& S'  \ar[dl]    \ar[dr]   &   \\
A    & &  X \cong S/H  \\ 
}
\end{xy}
\]
A natural question to address is the classification of quotient surfaces $S/H$. 

A first step in studying the quotient surfaces is to determine their numerical invariants. To this end we study the induced  action of the group $H$ on the cohomology groups of $S$. Recall that the the global sections of  $H^{1,0}(S)$ comes from the one forms on $A$ that we denote by $dx,dy$. Moreover, one of the two generators of the  global sections of  $H^{2,0}(S)$ can be identified with $dx \wedge dy$, and of course being $p_g(S)=2$  we have a global 2-forms $\omega$ not coming from $A$. To summarize this we can write
\begin{align*}
H^{2,0}(S) \cong \langle dx \wedge dy, \omega \rangle, & & H^{1,0}(S) \cong \langle dx, dy  \rangle.
\end{align*}
 
Let us denote by $g$ the generator of the cyclic group $G$, which has order $2$ or $4$ according to the three cases of Proposition \ref{prop:AutA}. The same proposition describes completely the induced action on the generators of the cohomology groups in each case.  In particular, we have 
\begin{align*}
dx \mapsto g^*dx= \begin{cases}  -dx, \\idx
\end{cases} &  &dy \mapsto g^*dy=dy.
\end{align*}

If $H$ is trivial or $H \cong <\sigma>$ then we know that the quotients are respectively $S$ or birational to $A$. Else $H \cap G \neq \{1\}$ and this  yields at once that $q(X) =1$. More precisely, $H^1(X)\cong <dy>$,  and thus we have proved Corollary \ref{cor_main}. We can remark already that no quotient $X$ can be a K3 surface.  Moreover, we have
 \[
g^*(dx \wedge dy) = \begin{cases}  -dx\wedge dy , \\idx \wedge dy.
\end{cases} 
\]

Therefore we have that $p_g(X)$ is either $0$ or $1$ according to the $H$-invariance of $\omega$ .


\section{Towards the Mumford-Tate conjecture}\label{sec_moreauto}

This section is devoted to explain the strategy used up to know for proving the Mumford Tate conjecture for surfaces with $p_g(S)=q(S)=2$ and of maximal Albanese dimension.

Let $S$ be a smooth projective complex surface
with invariants $p_g(S) = q(S) = 2$,
and assume that the Albanese morphism $\alpha \colon S \to A$
is surjective.
We can make the following general observations (see also \cite{CoPe20}). It holds:
\begin{enumerate}
 \item The induced map on cohomology
  $\alpha^* \colon H^*(A,\ZZ) \to H^*(S,\ZZ)$ is injective.
  The orthogonal complement
  $H^2_{\new} = H^*(A,\ZZ)^\perp \subset H^*(S,\ZZ)$
  is a Hodge structure of weight~$2$
  with Hodge numbers $(1,n,1)$,
  where $n = h^{1,1}(S) - 4$.
  Such a Hodge structure is said to be of \emph{K3~type}.
 \item Let $S'$ be a smooth projective complex surface
  with invariant $p_g(S') = 1$.
  Then Morrison~\cite{MorrIsog} showed that there
  exists a K3~surface~$X'$ together with
  an isomorphism $\iota' \colon H^2(S',\QQ)^\tra \to H^2(X',\QQ)^\tra$
  that preserves the Hodge structure, the integral structure,
  and the intersection pairing.
  (Here $(\_)^\tra$ denotes
  the \emph{transcendental} part of the Hodge structure,
  that is, the orthogonal complement of the Hodge classes.)
\end{enumerate}

We now look closely to our surfaces $S$ with $p_g(S)=q(S)=2$, for which we know that the Albanese map is a generically finite  cover. 

Then we have

\begin{prop} \label{main-theorem} 
 Let $S$ be a smooth projective complex surface
 with invariants $p_g(S) = q(S) = 2$,
 and assume that the Albanese morphism $\alpha \colon S \to A$
 is surjective. Then there exists a K3~surface~$X$
  and an isomorphism of Hodge structures
   $$\iota \colon (H^2_{\new}(S,\QQ))^\tra \to H^2(X,\QQ)^\tra.$$
\end{prop}

This is a direct  consequence of \cite{MorrIsog}. Notice that the surface $X$ is related only Hodge theoretically to $S$. Therefore, this is not enough to prove the conjecture, to this end we have to address the following question:

\smallskip
{\narrower\it\noindent
 Do there exist $X$ and~$\iota$  as above,
 such that $\iota$ is motivated in the sense of Andr\`e?
 \par}
\smallskip

Let us briefly explain and recall some facts on categories of motives, and for the reader convenience we state the motivic Mumford--Tate conjecture, for a more detailed introduction on the subject see \cite{Moonen, Moonenh1}. First we recall some facts about Chow motives and Andr\'e motives of surfaces. 
We do not need full generality, so let $K$ be a subfield of~$\CC$.

Given smooth and projective varieties $X$ and $Y$ over a field $K$ (i.e., objects in the category $\SmPr_{/K}$) of dimension $d_X$ and $d_Y$ respectively, a correspondence of degree $k$ from $X$ to $Y$ is an element $\gamma$ of $A^{d_X+k}(X \times Y)$. Then $\gamma$ induces a map $A^{\cdot}(X) \ra A^{\cdot+k}(Y)$ by the formula
\[
\gamma_*(\beta) := \pi_{2*}(\gamma \cdot \pi^*_1 (\beta)),
\]
where $\pi_1\colon X\times Y \ra X$ and $\pi_2\colon X\times Y \ra Y$ denote the projections.
The category $\ChMot$ of \emph{Chow motives} (with rational coefficients) over $K$ is defined
as follows:
\begin{itemize}
	\item the objects of $\ChMot$ are triples $(X, p, n)$ such that $X \in \SmPr_{K},\, p \in A^{d_X} (X \times X)$ is an idempotent correspondence (i.e. $p_* \circ p_* = p_*$) and $n$ is an integer;
	\item the morphisms in $\ChMot$ from $(X, p, n)$ to $(Y, q, m)$ are  correspondences $f\colon  X \ra Y$
of degree $n-m$, such that $f \circ p = f = q \circ f$. 
\end{itemize}
We recall that $\ChMot$ is an additive, $\QQ$-linear, pseudoabelian category, see \cite[Theorem~1.6]{Scholl}.

We consider from here only the cases in which we are interested hence let us suppose $K=\CC$, There exists a functor 
\[
h \colon \SmPr_{/\CC}^{\opp} \to \ChMot\ \ \mbox{such that } h:X\mapsto (X,\textrm{id},0)
\]
from the opposite category of smooth projective varieties over~$\CC$
to the category of Chow motives.

We denote also with $A^i(M)$ the $i$-th Chow group
of a motive $M \in \ChMot$.
In general, it is not known whether
the K\"unneth projectors~$\pi_i$ are algebraic,
so it does not (yet) make sense to speak of the summand $h^i(X) \subset h(X)$
for an arbitrary smooth projective variety $X/\CC$.
However, a so-called Chow--K\"unneth decomposition does exist
for curves~\cite{Manin_motive},
for surfaces~\cite{Murre_motsurf},
and for abelian varieties~\cite{DenMur}.
For algebraic surfaces there is in fact the following theorem,
which strengthens the decomposition of the Chow motive.
Statement is copied from \cite[Theorem 2.2]{Lat19}.

\begin{thm} \label{ChSurf} 
Let $S$ be a smooth projective surface over~$\CC$.
There exists a self-dual Chow--K\"unneth decomposition $\{\pi_i\}$ of~$S$,
with the further property that there is a splitting
\[
\pi_2 = \pi_2^\alg + \pi_2^\tra \quad \in A^2(S \times S)
\]
in orthogonal idempotents, defining a splitting
$h^2(S) = h^2_\alg(S) \oplus h^2_\tra(S)$ with Chow groups
\[
A^i(h^2_\alg(S)) =
\begin{cases}
					\textrm{NS}(S) & \text{if $i = 1$,}\\
					0 & \text{otherwise,}
				\end{cases} \qquad
\text{and} \quad A^i(h^2_\tra(S)) =
				\begin{cases}
					A^2_{\AJ}(S) & \text{if $i = 2$,}\\
					0 & \text{otherwise.}
				\end{cases}
				\]
				Here $A^2_{\AJ}(S)$ denotes the kernel of the Abel--Jacobi map.
\end{thm}
\begin{proof}
For the proof see \cite[Theorem 2.2]{Lat19} and references therein. 
\end{proof}

The idea to construct such K3~surface $X$ is to exploit the automorphism group $G$ of $S$ and prove that a quotient $S/H$  by some subgroup $H \leq G$  is birational to $X$. Of course this will give us a weak answer to the previous question.  Nevertheless,  it will suffices to prove that  -- using  the notion of \emph{motivated cycles}
introduced by Andr\'e~\cite{AndrMoti} --  there exist $X$ and~$\iota$  as above,
 such that $\iota$ is motivated.

\smallskip

To speak about motivic Mumford--Tate conjecture we need to introduce the notion of {\it motivated cycles} (for a brief introduction see e.g. \cite[Section 3.1]{Moonen}).

\begin{defin}
	Let $K$ be a subfield of~$\CC$,
 and let $X$ be a smooth projective variety over~$K$.
A class $\gamma$ in $H^{2i}(X)$ is called
a \emph{motivated cycle} of degree~$i$
if there exists an auxiliary smooth projective variety~$Y$ over~$K$
such that $\gamma$ is of the form $\pi_*(\alpha \cup \star\beta)$,
where $\pi \colon X \times Y \to X$ is the projection,
$\alpha$ and~$\beta$ are algebraic cycle classes in $H^*(X \times Y)$,
and $\star\beta$ is the image of~$\beta$ under the Hodge star operation.
\end{defin}

Every algebraic cycle is motivated,
and under the Lefschetz standard conjecture the converse holds as well.
The set of motivated cycles naturally forms a graded $\QQ$-algebra.
The category of motives over~$K$, denoted~$\Mot_K$,
consists of objects $(X,p,m)$,
where $X$ is a smooth projective variety over~$K$,
$p$ is an idempotent motivated cycle on $X \times X$,
and $m$ is an integer.
A morphism $(X,p,m) \to (Y,q,n)$
is a motivated cycle~$\gamma$ of degree $n-m$ on $Y \times X$
such that $q \gamma p = \gamma$.
We denote with $\HH(X)$ the object $(X,\Delta,0)$,
where $\Delta$ is the class of the diagonal in $X \times X$.
The K\"unneth projectors $\pi_i$ are motivated cycles,
and we denote with $\HH^i(X)$ the object $(X,\pi_i,0)$.
Observe that $\HH(X) = \bigoplus_i \HH^i(X)$.
This gives contravariant functors $\HH(\_)$ and $\HH^i(\_)$
from the category of smooth projective varieties over~$K$ to~$\Mot_K$.
			
\begin{thm} 
\label{mot-props}
	The category $\Mot_K$ is Tannakian over~$\QQ$,
	semisimple, graded, and polarised.
	Every classical cohomology theory of smooth projective varieties over~$K$
	factors via~$\Mot_K$.
\begin{proof}
	See th\'eor\`eme~0.4 of~\cite{AndrMoti}.
\end{proof}
\end{thm}

\begin{defin} 
	Let $K$ be a subfield of~$\CC$.
	An \emph{abelian motive} over~$K$ is an object
	of the Tannakian subcategory of $\Mot_K$
	generated by objects of the form $\HH(X)$
	where $X$ is either an Abelian variety
	or $X = \Spec(L)$ for some finite extension
	$L/K$, with $L \subset \CC$.
										
	We denote the category of abelian motives over~$K$ with~$\AbMot_K$.
\end{defin}
									
Finally we need the following theorem
									
\begin{thm} 
	
	\label{hodge-is-motivated}
	The Hodge realization functor $H(-)$ restricted to the subcategory of abelian motives is a full functor.
\begin{proof}
See th\'eor\`eme~0.6.2 of~\cite{AndrMoti}.
\end{proof}
\end{thm}

By Theorem  \ref{mot-props}, the singular cohomology and $\ell$-adic cohomology functors
factor via~$\Mot_K$.
This means that if $M$ is a motive,
then we can attach to it a Hodge structure~$\HB(M)$
and an $\ell$-adic Galois representation $\Hl(M)$.
The Artin comparison isomorphism between singular cohomology
and a $\ell$-adic cohomology extends to an isomorphism of vector spaces
$\Hl(M) \cong \HB(M) \otimes \QQl$ that is natural in the motive~$M$.
												
We can state the motivated Mumford--Tate conjecture following \cite[Section 3.2]{Moonen}. We shall write $\GB(M)$ for the Mumford--Tate group $\GB(\HB(M))$.
Similarly, we write $\Gl(M)$ (resp.~$\Glc(M)$)
for $\Gl(\Hl(M))$ (resp.~$(\Glc(\Hl(M)$) for the Tate group.
The Mumford--Tate conjecture extends to motives:
for the motive~$M$ it asserts that the comparison isomorphism
$\Hl(M) \cong \HB(M) \otimes \QQl$
induces an isomorphism
$$\Glc(M) \cong \GB(M) \otimes \QQl.$$

The discussion we have given here  enable us to prove the following.
\begin{prop}
Let $S$ be a surface of general type as above  if there exists a subgroup $H$ of the automorphisms group~$G$ of $S$,
   and there exist a K3~surface~$X$ birational to $S/H$  and~$\iota$ as in (H) is \emph{motivated} (in the sense of Andr\'e) then
   the Tate and Mumford--Tate conjectures hold for~$S$.
   That is the Tate and Mumford--Tate conjectures hold for those~$S$
   that are deformation equivalent to such a surface $S$.
\end{prop}

\begin{proof} The proof of this theorem is contained in\cite[Section 5]{CoPe20}. We illustrate only the demonstration strategy in the realm of motives. 
												
The main idea in \cite{CoPe20} is that for surfaces $S$ with $p_g=2$ it is sometimes  possible to decompose the weight $2$ Hodge structure into two Hodge substructures of K3 type and see that  these Hodge substructures  are indeed the Hodge structures of either Abelian surfaces or K3 surfaces which are (birational) quotients of $S$. This geometric construction makes possible to consider the theory of motivated cycles introduced by Andr\'e, and to decompose the motive of $S$ into two abelian motives of K3 type. For these motives the Mumford--Tate conjecture is known. This, together with the main results of \cite{Com16} and \cite{Com19} allows to prove the Mumford-Tate conjecture for $S$. In \cite{CoPe20} it was used the fact that the surfaces studied are of maximal Albanese dimension hence there is naturally an Abelian surface as a quotient surface. 

\end{proof}

Let us summarize here, in the form of a table, the classification of the minimal surfaces of general type with $p_g=q=2$ and of maximal Albanese dimension. Moreover,  for each family we will indicate whether the Mumford--Tate conjecture has been proved or not.

\begin{table}[h]
 \begin{center}
 \begin{tabular}{cccccccccc}
  \textnumero &$K^2_S$ 
  & $\deg(\alpha)$ & $\#$ & $\dim$ &
  Name & \textsc{mtc} & \textsc{pq/{ms}} & Reference \\
  \hline
  1 & $8$  & $2$ & $2$ & $0^2$ &
  &
  & No & \cite{PRR} \\
  2 & $8$  & $ \{ 2,4,6\}$ & $4$ & $3^3,4$ &
 SIP &
  \checkmark & Yes & \cite{pe11} \\
  3 & $7$  & $3$ & $1$ & $3$ & PP7
  &
  \checkmark & Yes & \cite{PiPol17,CF18} \\
  4 & $7$  & $2$ & 3 & $3$ & 
  &
  & ? & \cite{rito, MatteoRoberto} \\ 
  \hline 
  5 & $6$  & $4$ & $1$ & $4$ & PP4
  &
  \checkmark & Yes & \cite{PP14} \\
    6 & $6$  & $3$ & $?$ & {$3$} & AC3
  & 
  & ? & \cite{AC,CS} \\
  7 & $6$  & $2$ & $1$ & {$3$} & PP2
  & 
  & ? & \cite{PP} \\
  8 &   $6$  &   $2$ &   $2$ &   $4^2$ & PP2
  &    \checkmark
  &   Yes & \cite{PP} \cite{Pi17} \\
  9 & $5$  & $3$ & $1$ & $4$ &
  CHPP &
  \checkmark & Yes & \cite{CH,PP13b} \\
  10 & $4$  & $2$ & $1$ & $4$ &
  Catanese &
  \checkmark & Yes & \cite{pe11,CML02}
 \end{tabular}
 \end{center}
  \caption{State of the art of the classification
  of minimal complex algebraic surfaces with invariants $p_g = q = 2$ and with maximal Albanese dimension.}

 \label{state_of_the_art}
\end{table}

  In Table \ref{state_of_the_art} have indicated, where possible,
  the number of families~($\#$)
  and the dimensions of the irreducible component containing them~($\dim$).
  Moreover, we point out if some members of the family
  are product-quotient surfaces~(\textsc{pq}) or mixed surfaces~(\textsc{ms}), in particular for some PP2 surfaces the proof is given by the second author in \cite{Pi17}, while for the remaining ones is unknown.
  In the last column,
  we give references to more detailed descriptions of the class.
  Finally,
  we put a checkmark in the column~\textsc{mtc}
  if the strategy given in \cite{CoPe20} is enough to prove the Tate and Mumford--Tate conjectures for a class. As one can see, up to now, it is the only way to prove  Mumford--Tate conjecture for these surfaces. 
  
Finally, we explain here the meaning of the half-line between the surfaces with $K^2_S \geq 7$ and $K^2_S \leq 6$. This line is due to the classification theorem of \cite{DJZ}, who recently proved that the classification of surfaces with $p_g=q=2$ and $K^2_S \leq 6$ is complete. 


\bigskip

Matteo Penegini, Universit\`a degli Studi di Genova, DIMA Dipartimento di Matematica, I-16146 Genova, Italy \\
\emph{e-mail} \verb|penegini@dima.unige.it|

\medskip

Roberto Pignatelli Universit\`a degli Studi di Trento, Dipartimento di Matematica,
I-38123 Trento,  Italy \\
\emph{e-mail} \verb|Roberto.Pignatelli@unitn.it|



%


\end{document}